\documentclass[12pt]{article}
 
\usepackage{graphicx}%
\usepackage{multirow}%
\usepackage{amsmath,amssymb,amsfonts}%
\usepackage{amsthm}%
\usepackage{mathrsfs}%
\usepackage[title]{appendix}%
\usepackage{xcolor}%
\usepackage{textcomp}%
\usepackage{manyfoot}%
\usepackage{booktabs}%
\usepackage{algorithm}%
\usepackage{algorithmicx}%
\usepackage{algpseudocode}%
\usepackage{listings}%

\theoremstyle{thmstyleone}%
\newtheorem{theorem}{Theorem}
\newcommand{\nbc}{\mbox{{\textmd {\it C}}}}
\newcommand{\nba}{\mbox{{\textmd {\it A}}}}

\newcommand{\bb}{\mbox{\boldmath$b$}}
\newcommand{\be}{\mbox{\boldmath$e$}}

\newcommand{\br}{\mbox{\boldmath$r$}}

\newcommand{\bv}{\mbox{\boldmath$v$}}

\newcommand{\bx}{\mbox{\boldmath$x$}}
\newcommand{\by}{\mbox{\boldmath$y$}}
\newcommand{\bz}{\mbox{\boldmath$z$}}

\newcommand{\nul}{{\cal N}}
\newcommand{\ran}{{\cal R}}
\newcommand{\rank}{\mbox{rank}}
\newcommand{\real}{{\bf R}}

\newcommand{\rmn}{ {\bf R}^{m \times n} }
\newcommand{\Rm}{ {\bf R}^m }
\newcommand{\rn}{{\bf R}^n}

\newcommand{\rnm}{{\bf R}^{n \times m}}
\newcommand{\rnn}{{\bf R}^{n \times n}}

\newcommand{\trans}{{\mbox{\scriptsize T}}}

\def\vector#1{\mbox{\boldmath $#1$}}

\theoremstyle{thmstyletwo}%

\theoremstyle{thmstylethree}%

\begin{document}

\title{Right preconditioned GMRES for arbitrary singular systems}









\author{Kota Sugihara\thanks{E-mail: kouta.sugihara@gmail.com} \,
and Ken Hayami\thanks{Professor Emeritus, National Institute of Informatics, and 
The Graduate University for Advanced Studies (SOKENDAI),
Tokyo, Japan, E-mail: hayami@nii.ac.jp}
}


\maketitle

\abstract{Brown and Walker (1997) showed that GMRES determines a least squares solution of 
$\nba \bx = \bb$
where  $ \nba \in \rnn $
without breakdown for arbitrary $ \bb, \bx_0 \in \rn $ if and only if $\nba$ is range-symmetric, 
i.e. $ \ran(\nba^\trans) = \ran(\nba) $, where $ \nba $ may be singular and $\bb $ may not be in the 
range space $\ran(\nba)$ of $\nba$.
In this paper, we propose applying GMRES to $ \nba \nbc \nba^\trans \bz = \bb $, where $ \nbc \in \rnn $ is 
symmetric positive definite. This determines 
a least squares solution $ \bx = \nbc \nba^\trans \bz$ of $ \nba \bx = \bb $ without breakdown for 
arbitrary (singular) matrix $\nba \in \rnn$ and $\bb \in \rn$.
To make the method numerically stable, we propose using the pseudoinverse with an appropriate 
threshold parameter to suppress the influence of tiny singular values when solving 
the severely ill-conditioned Hessenberg systems which arise in the Arnoldi process of 
GMRES when solving inconsistent range-symmetric systems.
Numerical experiments show that the method taking $\nbc$ to be the identity matrix and the inverse matrix of 
the diagonal matrix whose diagonal elements are the diagonal elements of $\nba^{\trans}\nba$ gives 
a least squares solution even when $\nba$ is not range-symmetric, including the case when 
$ {\rm {\textmd {index}}}(\nba) >1$.}

\section{Introduction}\label{sec1}

The Generalized Minimal Residual (GMRES) method \cite{SS} is a robust and efficient Krylov subspace 
iterative method for systems of linear equations
\begin{equation}\label{lineq}
A \bx = \bb
\end{equation}
 where $ A \in \rnn$ is nonsingular and may be nonsymmetric, and $\bx, \bb \in \rn$. 

Abstractly, GMRES begins with an initial approximate solution $ \bx_0 \in \rn$ and initial residual $ \br_0 = \bb - A \bx_0 $ and 
characterizes the $k$th approximate solution as $\bx_k = \bx_0 + \bz_k$, where $\bz_k$ solves
\begin{equation}
 \min_{ \bz \in {\cal K}_k } \| \bb - A ( \bx_0 + \bz ) \|_2 
  = \min_{ \bz \in {\cal K}_k } \| \br_0 - A \bz \|_2 .
 \label{GMRES}
\end{equation}
Here, ${\cal K}_k$ is the $k$th Krylov subspace determined by $A$ and $\br_0$, defined by
\[ {\cal K}_k \equiv {\rm span} \{ \br_0, A\br_0, \ldots, A^{k-1} \br_0 \}. \]

The implementation given in \cite{SS} is as follows.
\\ \\
{\bf Algorithm 1: GMRES}\\ \\
Choose $\bx_0$. 
$ \br_0 = \bb - A \bx_0 $ 
$ \bv_1 = \br_0 / ||\br_0||_2 $ \\
For $ j = 1, 2, \cdots $ until satisfied do \\
\hspace{5mm} $ h_{i,j}=(\bv_i,A \bv_j )\hspace{4mm}(i=1,2,\ldots,j) $ \\ 
\hspace{5mm}
$ {\displaystyle \hat{\bv}_{j+1} = A \bv_j - \sum_{i=1}^j h_{i,j} \bv_i } $ \\
\hspace{5mm} $ h_{j+1,j} =||\hat{\bv}_{j+1} ||_2 $.
\hspace{4mm}If $ h_{j+1,j} =0, $ goto $ \ast $. \\
\hspace{5mm} $ \bv_{j+1} = \hat{\bv}_{j+1} / h_{j+1,j} $ \\
End do \\
$ \ast \, k:=j $ \\
Form the approximate solution \\ 
\hspace{5mm} $ \bx_k = \bx_0 + [ \bv_1,\ldots,\bv_k] \by_k $ \\
where $ \by = \by_k $ minimizes $  ||\br_k||_2 = ||\beta\be_1-\overline{H}_k \by ||_2 . $
\begin{equation}
\nonumber
\end{equation}

Here, $\vector{r}_{k} = \vector{b} - A\vector{x}_{k}$, and $ \overline{H}_k = [ h_{i,j} ] \in {\bf R}^{ (k+1) \times k}$ is a 
Hessenberg matrix, i.e., $h_{i,j}=0$ for $i>j+1$. \, \,
$\beta={||\br_0 ||_2} $ \, and \, $ \be_1 = [1,0,\ldots,0]^\trans \in {\bf R}^{k+1} $.

There are two kinds of definitions for the breakdown of GMERS.
One definition is that
GMRES breaks down at the $k$th iteration when 
${\rm dim}A({\cal K}_k) < {\rm dim}{\cal K}_k$ or ${\rm dim}{\cal K}_k < k$ \cite{BW}.
The other definition is that GMRES breaks down at the $k$th iteration when $ h_{k+1,k} = 0$ in Algorithm 1.
The equivalence between these definitions for the breakdown of GMRES is discussed in \cite{MH}. 
In this paper, we will use the latter definition.

When $ A $ is singular, one may still apply GMRES to the least squares problem
\begin{equation}
 \min_{\bx \in \rn} \| \bb - A \bx \|_2 .
 \label{lstsq}
\end{equation} 
However, Brown and Walker \cite{BW} (Theorem 2.4) showed that GMRES determines a least squares solution of (\ref{lstsq}) without 
breakdown for all $ \bb $ and initial approximate solution $ \bx_0 $ if and only if $ A $ is range-symmetric (EP), 
that is $ \ran(A)=\ran(A^\trans)$, where $ \ran(A) $ denotes the range space of $ A $. (See also \cite{HS} for an 
alternative proof.)

In order to overcome this problem, Reichel and Ye \cite{RY} proposed the breakdown-free GMRES, which is designed to expand 
the solution Krylov subspace when GMRES breaks down. However, they do not give a theoretical justification that 
this method always works.

On the other hand, for the least squares problem 
\begin{equation}
 \min_{\bx \in \rn} \| \bb - A \bx \|_2
 \label{lstsq2}
\end{equation} 
where $ A \in \rmn $, Hayami, Yin and Ito \cite{HYI} proposed the AB-GMRES method which applies GMRES to 
\begin{equation}
 \min_{\bz \in \Rm} \| \bb -AB \bz \|_2 
 \label{AB}
\end{equation}
where $ B \in \rnm $ satisfies $ \ran (AB) = \ran (A) $.

If we let $ m=n $, and $ B=CA^\trans $, where $ C \in \rnn $ is symmetric positive-definite, then $ \ran(A) = \ran (AB) $
 holds (\cite{HYI}, 
Lemma 3.4).
Note also that $ \ran ( (AB)^\trans )= \ran (AB)$, i.e. $ AB $ becomes range-symmetric, so that the AB-GMRES determines a least 
squares solution of (\ref{AB}) for arbitrary $ \bb \in \Rm$. Note also that $B=CA^\trans$ can be regarded as a preconditioner.

However, when $ \bb \notin \ran (A)$, even when $\ran ( A^\trans )= \ran (A) $, the least squares problem (\ref{GMRES}) 
in GMRES may 
become dangerously ill-conditioned before a least squares solution of  (\ref{lstsq}) is obtained (\cite{BW}, Theorem 2.5).

In order to overcome this problem, we proposed using pseudoinverse to suppress the influence of tiny singular values when solving 
the severely ill-conditioned Hessenberg system which arises when solving inconsistent range-symmetric systems \cite{SHL}. 
In this paper, we further optimize this method using an appropriate threshold for the pseudoinverse and apply the method for 
solving arbitrary singular systems which may be inconsistent. Numerical experiments with inconsistent systems with 
non-range-symmetric 
$A$ with 
index($A$) $\geq$ 1, show the validity of the proposed method.

The rest of the paper is organized as follows. In section 2, we explain the definition of index, GP, EP matrices and
their applications.
In section 3, we review the theory for GMRES on singular systems. 
In section 4, we propose the right preconditioned GMRES for arbitrary singular systems. In section 5, we propose using 
pseudoinverse with an appropriate threshold to solve the severely ill-conditioned Hessenberg systems arising from the Arnoldi 
process in GMRES. We present numerical experiment results of the proposed method for GP systems in section 6,
and for index 2 systems in section 7.
In section 8, we compare the proposed method with the left preconditioned GMRES applied to the first kind normal equation 
which is equivalent to the least squares problem. Section 9 concludes the paper.

\section{GP and EP matrices and their applications}
index($A$) of $A \in \rnn$ denotes the smallest nonnegative interger $i$  such 
that $\rank(A^{i}) = \rank(A^{i+1})$ \cite{CM}
(Definition 7.2.1) and it is thus equal to the size of the largest Jordan block 
corresponding to the zero eigenvalue of $A$ \cite{OL} (section 3). 
If $\ran(A) \cap \nul(A) = {\vector{0}}$,
 $A$ is called a GP (group) matrix \cite{HaSp} (section 1). 
(A group matrix is a matrix which has a group inverse.)  
If $A$ is singular, $A$ is GP if and only if ${\rm index}(A) = 1$.
The GP matrix arises, for example, from the finite difference discretization of a convection diffusion 
equation with Neumann boundary condition 
\begin{eqnarray}\label{neueq}
\Delta u + d \frac{\partial u}{\partial x_{1}} & = & x_{1} + x_{2},~~\vector{x}=(x_{1}, x_{2}) \in \Omega \equiv [0,1] \times [0,1],  \\
\frac {\partial u(\vector{x})}{\partial \nu} & = & 0 ~~for~~ \vector{x} \in \partial \Omega ,
\end{eqnarray}
as in \cite{BW}.

Assume $A$ arises from the finite difference discretization of the above equation. 
Then, $\nul(A^{\trans}) \neq \nul(A)$ and 
$\ran(A) \cap \nul(A) = \{\vector{0}\}$ hold. Then, A is a GP matrix.

GP matrices also arise in the analysis of ergodic homogeneous finite Markov chains \cite{FH}.

When $\ran(A^\trans) = \ran(A)$, $A$ is called range-symmetric or EP (Equal Projectors) \cite{CM}. Note that, since
$\nul(A^\trans) = \ran(A)^{\perp}$, $\ran(A^\trans) = \ran(A)
\Leftrightarrow \nul(A^\trans) = \nul(A) \Leftrightarrow \ran(A) \perp \nul(A)$.
Hence, $\ran(A^\trans) = \ran(A) \Rightarrow \ran(A) \cap \nul(A) = \{\vector{0}\}$.
That is, an EP matrix is a GP matrix. An EP matrix arises, for instance, from the finite difference
approximation
of the convection diffusion equation (\ref{neueq}) with a periodic boundary condition 
\begin{eqnarray*}
u(x_{1}, 0) = u(x_{1}, 1), ~~x_{1} \in [0,1], \\
u(0, x_{2}) = u(1, x_{2}), ~~x_{2} \in [0,1],
\end{eqnarray*}
as in \cite{BW}.

\section{Convergence theory of GMRES}
Consider the system of linear equations (\ref{lineq})
and the least squares problem (\ref{lstsq}).
(\ref{lineq}) is called consistent 
when $\vector{b}\in R(A)$, and inconsistent otherwise.
Brown and Walker \cite{BW} showed that GMRES determines a least squares solution of 
(\ref{lstsq}) without breakdown for arbitrary $ \bb, \bx_0 \in \rn $ if and only if $A$ is range symmetric (EP), 
i.e. $\ran(A^\trans) = \ran(A)$, where $ A \in \rnn $ may be singular and $\bb $ may not be  in the range space $\ran(A)$ of $A$.

When $\ran(A^\trans) \ne \ran(A)$, there exist $\vector{x}_{0}$ and $\vector{b}$ such that 
GMRES breaks down without giving a least squares solution of (\ref{lstsq}) (\cite{BW, HS}).

In \cite{BW}, it was also pointed out that even if $\ran(A^\trans) = \ran(A)$, if $\vector{b} \notin \ran(A)$,
the least squares problem (\ref{GMRES}) becomes very ill-conditioned. (The condition number of the Hessenberg matrix
arising in each iteration of GMRES becomes very large.)

If $A$ is a GP matrix, GMRES determines a solution of (\ref{lineq}) for all $\vector{b}\in \ran(A)$ and 
for all initial vector $\vector{x}_{0}$ (\cite{BW}, Theorem 2.6). Moreover,  
GMRES determines a solution of (\ref{lineq}) for all $\vector{b}\in \ran(A)$ and 
for all initial vector $\vector{x}_{0}$ if and only if $A$ is a GP matrix (\cite{HS}, Theorem 2.8). 

Even if $A$ is a GP matrix and $\vector{b} \in \ran(A)$, 
it was reported in \cite{MR} that
 the least squares problem (\ref{GMRES})
can become very ill-conditioned before breakdown in finite precision arithmetic.
If ${\rm index}(A) > 1$, there exists a $\vector{b} \in \ran(A)$ such that GMRES breaks down at iteration step 1 without
giving a solution of (\ref{lineq}) (\cite{HS}, Theorem 2.8). 

Hence, we propose GMRES using right preconditioning and pseudoinverse to overcome the difficulty of solving the least squares problems 
(\ref{lstsq}) with coefficient matrices whose index is greater than or equal to 1.

\section{Right preconditioned GMRES for singular systems}
For the least squares problem (\ref{lstsq2}) 
where $ A \in \rmn $, Hayami, Yin and Ito \cite{HYI} proposed the AB-GMRES method which applies GMRES to (\ref{AB})
where
$ B \in \rnm $ satisfies $ \ran (AB) = \ran (A) $.
Note that 
$\displaystyle  \min_{\bz \in \Rm} \| \bb -AB \bz \|_2 = \min_{\bx \in \rn} \| \bb -A \bx \|_2$  holds
for all $\vector{b} \in \rn$ if and only if $ \ran(A) = \ran (AB) $ (\cite{HYI}, Theorem 3.1).
If we let $ m=n $, and $ B=CA^\trans $, where $ C \in \rnn $ is symmetric positive definite, then $ \ran(A) = \ran (AB) $ 
holds (\cite{HYI}, Lemma 3.4).
Note also that $\ran ((AB)^\trans) = \ran(AB) $, i.e. $ AB $ becomes 
range-symmetric, so that the AB-GMRES determines a least 
squares solution of (\ref{AB}) for arbitrary $ \bb \in \Rm$. Note also that $B=CA^\trans$ can be regarded as a right preconditioner.
Using this right preconditioning, we transform (\ref{lstsq}), where index($A$) is greater than or equal to 1, 
to a least squares problem (\ref{AB}) with $ B=CA^\trans $, where $ C \in \rnn $ is symmetric positive definite.
Since $\ran(A) = \ran (AB)$ holds, GMRES determines a least squares solution $\vector{z}$ of (\ref{AB}) without breakdown
and $\vector{x} = B\vector{z}$ is a least square solution of ($\ref{lstsq2}$).
However, when $ \bb \notin \rn $, even when $\ran ( A^\trans ) = \ran(A)$, the least squares problem (\ref{GMRES}) in GMRES may 
become dangerously ill-conditioned before a least squares solution of  (\ref{lstsq}) is obtained (\cite{BW}, Theorem 2.5).
Thus, in the next section, we propose using the pseudoinverse
with a threshold parameter to overcome this problem.

Let $A\in \rnn$ and $B = CA^{\trans}$, where $C \in \rnn$ is symmetric positive definite.
Let the singular values of $AC^{\frac{1}{2}}$ be $\sigma_{i}(1\leq i \leq n)$. 
Hayami, Yin, and Ito showed the following. \cite{HYI}
\begin{theorem}\label{thab}
The residual $\vector{r} = \vector{b} - A\vector{x}$, achieved by the $k$th step of AB-GMRES satisfies 
\begin{eqnarray*}
 \|\vector{r}_{k}{\mid}_{\ran(A)}\|_{2} \leq 2\bigg(\frac{\sigma_{1} - \sigma_{r}}{\sigma_{1} + \sigma_r}\bigg)^{k} 
\|\vector{r}_{0}{\mid}_{\ran(A)}\|_{2}.
\end{eqnarray*}
\end{theorem}

Here, let the singular values of $A$ be $\sigma_{i}(A) (i=1,..,r) ({\rm rank}(A) = r)$. 
Let $\kappa(A) = \frac{\sigma_{1}(A)}{\sigma_r(A)}$. Then, from Theorem \ref{thab}, the following theorem holds.
\begin{theorem}\label{atrab}
The residual $\vector{r} = \vector{b} - A\vector{x}$, achieved by the $k$th step of AB-GMRES satisfies 
\begin{eqnarray*}
\|A^\trans\vector{r}_{k}\|_{2} &\leq&
2\kappa(A)\bigg(\frac{\sigma_{1} - \sigma_{r}}{\sigma_{1} + \sigma_r}\bigg)^{k} 
\|A^{\trans}\vector{r}_{0}\|_{2}
\end{eqnarray*}
\end{theorem}

\begin{proof}
\begin{eqnarray*}
A^\trans \vector{r}& =& A^\trans(\vector{r}{\mid}_{\ran(A)} + \vector{r}{\mid}_{{\ran(A)}^{\perp}}) \\
                   & = & A^\trans\vector{r}{\mid}_{\ran(A)} + A^\trans\vector{r}{\mid}_{\nul(A^{\trans})} \\
                   & = & A^\trans\vector{r}{\mid}_{\ran(A)}
\end{eqnarray*}
From Theorem \ref{thab}, 
\begin{eqnarray*}
\|A^\trans\vector{r}_{k}\|_{2}   & = & \|A^\trans\vector{r}_{k}{\mid}_{\ran(A)} \|_{2} \\  
& \leq & \sigma_{1}(A^{\trans}) \times 2\bigg(\frac{\sigma_{1} - \sigma_{r}}{\sigma_{1} + \sigma_{r}}\bigg)^{k} 
\|\vector{r}_{0}{\mid}_{\ran(A)}\|_{2} \\
 & = & \sigma_{1}(A) \times 2\bigg(\frac{\sigma_{1} - \sigma_{r}}{\sigma_{1} + \sigma_{r}}\bigg)^{k} 
 \|\vector{r}_{0}{\mid}_{\ran(A)}\|_{2}
\end{eqnarray*}
On the other hand, 
\begin{eqnarray*}
\|A^{\trans}\vector{r}_{0}\|_{2} & = &   \|A^{\trans}\vector{r}_{0}{\mid}_{\ran(A)}\|_{2}\\
                             & \geq & \sigma_{r}(A)\|\vector{r}_{0}{\mid}_{\ran(A)}\|_{2}
\end{eqnarray*}
Then,
\begin{eqnarray*}
\|A^\trans\vector{r}_{k}\|_{2} &\leq&
2\kappa(A)\bigg(\frac{\sigma_{1} - \sigma_{r}}{\sigma_{1} + \sigma_r}\bigg)^{k} 
\|A^{\trans}\vector{r}_{0}\|_{2}
\end{eqnarray*}
holds.
\end{proof}

\section{GMRES using pseudoinverse and reorthogonalization}
When \\ $\vector{b} \notin \ran(A)$, even when $\ran (A) = \ran (A^\trans)$,
the GMRES solution in finite precision arithmetic may become inaccurate when the least squares problem (\ref{GMRES}) in GMRES 
becomes dangerously ill-conditioned before a least squares solution of (\ref{lstsq}) is obtained.
Thus, we proposed the GMRES using pseudoinverse in order to improve the accuracy of the GMRES solution \cite{SHL}.
Let $H_{k+1,k}$ be the Hessenberg matrix arising in the GMRES Algorithm 1.

\vspace{6pt}
\hspace{-13pt}{\bf Algorithm 2 : GMRES using pseudoinverse (essence) }
\vspace{4pt}
\\1: Compute $y = {H_{k+1,k}}^{\dag}\beta\vector{e}_{1}$ where
${H_{k+1,k}}^{\dag}$ is the pseudoinverse of $H_{k+1,k}$.
\\2: Compute the solution $\vector{x}_{k}=\vector{x}_{0}+V_{k}y$.

Here,  $y = {H_{k+1,k}}^{\dag}\beta\vector{e}_{1}$ is the minimum-norm solution 
of\\
$\min_{\vector{y}_{k} \in  \mathbb{R}^{k}}\|\beta\vector{e}_{1} - H_{k+1,k}\vector{y}_{k}\|_{2}$ \cite{Bjorck}.  

${H_{k+1,k}}^{\dag}$ is defined as follows.

\vspace{6pt}
\hspace{-13pt}{\bf Definition 3 : Pseudoinverse of $B$}
\vspace{4pt}
\\1: Let the singular value decomposition of $B$ be $B = U\Sigma V^{\rm T}$ where $U\in {\rm \mathbb{R}}^{m\times m}$ and 
$V \in {\rm \mathbb{R}}^{n\times n}$ are orthogonal matrices, $\Sigma \in {\rm \mathbb{R}}^{m\times n}$ is
the diagonal matrix whose diagonal elements are the singular values
$\sigma_{1} \geq ...\geq \sigma_{r} > 0$ of $B$, $r = {\rm rank}B$, $\sigma_{i} = 0$, $i = r + 1, ..., \min\{m,n\}$. 
\\2: Then, $B^{\dag} = V{{\Sigma}^{\dag}}U^{\rm T}$.
Here, 
$\Sigma^{\dag} \in {\rm \mathbb{R}}^{n\times m}$ is the diagonal matrix whose diagonal elements are
${\sigma_1}^{-1} \leq ... \leq {\sigma_r}^{-1}$, ${\sigma_{i}}^{\dag} = 0$, $i = r + 1, ..., \min\{m,n\}$.

We use pinv in MATLAB for computing the pseudoinverse.
pinv for the matrix $B \in {\rm \mathbb{R}}^{m\times n}$ is defined as follows.

\vspace{6pt}
\hspace{-13pt}{\bf Algorithm 4 : pinv in MATLAB }
\vspace{4pt}
\\1:  Let the singular value decomposition of $B$ be $B = U{\Sigma}V^{\rm T}$ as above.
\\2: Set the tolerance value $tol$. The diagonal elements of $\Sigma$ which are smaller than $tol$ are replaced by zero
to give
\begin{eqnarray*}\label{eq:Asvd1}
\left[
\begin{array}{cc}
\Sigma_{1} & 0 \\
0                                 &  0                            
\end{array}
\right].
\end{eqnarray*}
Then, let
\begin{eqnarray}\label{eq:Asvd}
\tilde{B} := [U_{1},U_{2}]\left[
\begin{array}{cc}
\Sigma_{1} & 0 \\
0                                 &  0                            
\end{array}
\right][V_{1},V_{2}]^{\rm T}
=U_{1}{{\Sigma}_{1}}{V_{1}}^{\rm T}.
\end{eqnarray} 
where $U=[U_{1},U_{2}], ~V=[V_{1},V_{2}]$.
\\3: $\tilde{B}^{\dag} :=V_{1}{\Sigma_{1}}^{-1}{U_{1}}^{\rm T}$.

In \cite{SHL}, we used the default value of the tolerance value $tol$ which is \\
$\max\{m,n\}\times {\rm eps}(\|B\|_{2})$ for $B \in {\rm \mathbb{R}}^{m\times n}$.
Here,
\begin{itemize}
\item $d = {\rm eps}(x)$, where $x$ has data type single or double, returns the positive 
distance $d$ from $|x|$ to the next larger floating-point number of the same precision as $x$.   
\end{itemize}

$\max\{m,n\}\times {\rm eps}(\|B\|_{2})$ is called the numerical rank of $B$ \cite{Bjorck}.

Here, let $\sigma_{1}(H_{k+1,k})$ be the largest singular value of $H_{k+1,k}$, and 
$\sigma_{k}(H_{k+1,k})$ be the smallest singular value of $H_{k+1,k}$.
In this paper, we use $\alpha \times \sigma_{1}(H_{k+1,k}) \\
(0 < \alpha \ll 1)$ as the tolerance value $tol$ 
to improve the convergence of GMRES using pseudoinverse with the default value of the tolerace $tol$.\\

In the numerical experiments, we used $10^{-6}, 10^{-8}, 10^{-10}, 10^{-11}, 10^{-12}$ for $\alpha$.

Here, we define $\vector{r}_{j} = \|\vector{b} - A\vector{x}_{j}\|_{2}$ where $\vector{x}_j$ is the $j$th iterate of 
GMRES. 
It was observed in \cite{SHL} that,
when $\vector{b} \notin \ran(A)$,   
$\displaystyle \frac{\|A^{\rm T}\vector{r}_{j}\|_{2}}{\|A^{\rm T}\vector{b}\|_{2}}$ of 
GMRES using pseudoinverse sometimes oscillates. This is probably because the column vectors of $V_{k}$ lose orthogonality.
Thus, we proposed reorthogonalizing the column vectors of $V_{k}$ in order to
prevent the oscillation in \cite{Liao,SHL}.

The reorthogonalized Arnoldi process in the GMRES is as follows.

\vspace{6pt}
\hspace{-13pt}{\bf Algorithm 5 : Reorthogonalized Arnoldi process in the GMRES}
\vspace{4pt}
\\1:~~$h_{i,j} = (\vector{v}_{i}, A\vector{v}_{j})~(i=1,2,...j)$ 
\\2:~~$\vector{w}=A\vector{v}_{j} - {\displaystyle {\sum_{i=1}^j}}h_{i,j}\vector{v}_{i}$
\\3:~~$\hat{\vector{v}}_{j+1} = \vector{w} - {\displaystyle {\sum_{i=1}^j}}(\vector{w}, \vector{v}_{i})\vector{v}_{i}$ 
\\4:~~$h_{j+1,j}=\|\hat{\vector{v}}_{j+1}\|_{2}$
\\5:~~If $h_{j+1,j} \ne 0$, $\vector{v}_{j+1} = \frac{\hat{\vector{v}}_{j+1}}{h_{j+1,j}}$
\vspace{4pt}

In Algorithm 4, line 3 is the reorthogonalization part.

\section{Numerical experiments for GP systems}\label{NumGP}
In this section, we evaluate the effectiveness of the right preconditioned GMRES using pseudoinverse 
and reorthogonalization for inconsistent GP systems.
Furthermore, we also evaluate the effectiveness of the right preconditioned GMRES using reorthogonalization for  
consistent GP systems.

To do so, we compare the performance and the convergence when solving
the following three linear systems (\ref{axb}), (\ref{aateq}), (\ref{acateq}) by numerical experiments.
 Here, assume $C \in \rnn$ is symmetric positive definite.
For (\ref{axb}), we use GMRES using reorthogonalization.
For (\ref{aateq}) and (\ref{acateq}), 
we use GMRES using reorthogonalization, and GMRES using pseudoinverse and reorthogonalization.
\begin{eqnarray}
A\vector{x} &= &\vector{b} \label{axb}\\
AA^{\trans}\vector{z} & = & \vector{b},~\vector{x}=A^{\trans}\vector{z} \label{aateq} \\
ACA^{\trans}\vector{z} & = & \vector{b},~\vector{x}=CA^{\trans}\vector{z} \label{acateq}
\end{eqnarray}

GMRES for (\ref{aateq}) is the AB-GMRES with $B = A^{\trans}$ and GMRES for (\ref{acateq})
is the AB-GMRES with $B = CA^{\trans}$ \cite{HYI}. AB-GMRES means the right preconditioned GMRES.

The initial iterate $\vector{x}_{0}$ is set to $\vector{0}$.
Let $\vector{r}_{k}=\vector{b}-A\vector{x}_{k}$ where $\vector{x}_{k}$ is the approximate solution 
at the $k$th step.
We judge the convergence of each method by $\displaystyle \frac{\|A^{\rm T}\vector{r}_{k}\|_{2}}{\|A^{\rm T}\vector{b}\|_{2}}$
for inconsistent systems, and $\displaystyle \frac{\|\vector{r}_{k}\|_{2}}{\|\vector{b}\|_{2}}$
for consistent systems.

Computation except for Algorithm 1 of GMRES using pseudoinverse
were done on a PC with Intel(R) Core(TM) i7-7500U 2.70 GHz CPU, Cent OS and double precision
floating arithmetic.
GMRES was coded in Fortran 90 and compiled by Intel Fortran.
The method to code GMRES using pseudoinverse is as follows.
Here, $H_{i,j}$ is the Hessenberg matrix and all the column vectors of $V_{k}$ form an orthonormal basis
generated by the Arnoldi process.
\begin{enumerate}
\item $H_{i,j}$ and $V_{k}$ are computed by Fortran 90.
\item Write $H_{i,j}$ and $V_{k}$ into the ascii formatted files by Fortran 90.
\item Read the files of $H_{i,j}$ and $V_{k}$ in MATLAB.
\item The pseudoinverse 
${\tilde{H}_{i,j}}^{~~\dag}$
and the solution 
$\vector{x}_{k}=\vector{x}_{0}+V_{k}{{\tilde{H}}_{k+1,k}}^{~~~~~~\dag}\beta\vector{e}_{1}$ 
are computed using pinv of MATLAB.
\end{enumerate} 
The version of MATLAB is R2018b.

The purpose of numerical experiments in this paper is to verify
how small the residual of the propsed method can become.
Furthermore, we consider that the estimation of the cost performance 
of the proposed method is the future work.
Therefore, the size of the test matrices in the numerical experiments was kept small.


The GP matrix $A \in {\bf R}^{128 \times 128}$ is as follows.
\begin{eqnarray*}
\left[
\begin{array}{cc}
A_{11} & A_{12} \\
0                                 &  0                            
\end{array}
\right].
\end{eqnarray*}
where $A_{11}, A_{12} \in {\bf R}^{64 \times 64}$.
Here, assume $J_{k}(\lambda) \in {\bf R}^{k \times k} $ is a square matrix of the form
\begin{eqnarray*}
\left[
\begin{array}{cccc}
\lambda & 1 & 0 & 0\\
0      & \lambda & 1 & 0 \\
0      & 0      & ... & 1 \\
0      &  0     & 0   & \lambda\\                            
\end{array}
\right]
\end{eqnarray*}
where $\lambda \in \real$.

\begin{eqnarray*}
A_{11} & = & \left[
\begin{array}{cc}
W & 0 \\
0                                 &  D                            
\end{array}
\right].
\end{eqnarray*}
where $W \in {\bf R}^{32 \times 32}$ is 
\begin{eqnarray*}
\left[
\begin{array}{cccc}
J_{2}(\alpha_{1}) & 0 & 0 & 0 \\
0                 &  J_{2}(\alpha_{2}) & 0 & 0 \\
0                 &        0           & ... & 0 \\  
0                 &        0           & 0 & J_{2}(\alpha_{16})  
\end{array}
\right]
\end{eqnarray*}
and $D$ is diagonal matrix whose $j$th diagonal element is $\beta_{j}$.

Here, $\alpha_{j}(j=1,2,...16) \in \real$, $\beta_{i}(i=1,2,...,32) \in \real$ is
as follows. 
\begin{eqnarray*}
\alpha_{1} = 1 , \alpha_{16} = 10^{-\rho},~ \alpha_{j} = \alpha_{16} + \frac{16-j}{15}(\alpha_{1} - \alpha_{16})\times 0.7^{j-1} \\
\beta_{1} = 1 , \beta_{32} = 10^{-\gamma},~ \beta_{i} = \beta_{32} + \frac{32-i}{31}(\beta_{1} - \beta_{32})\times 0.2^{i-1}
\end{eqnarray*}
Furthermore, $A_{12} \in {\bf R}^{64 \times 64}$ is
\begin{eqnarray*}
\left[
\begin{array}{cccc}
J_{2}(\beta_{1}) & 0 & 0 & 0 \\
0                 &  J_{2}(\beta_{2}) & 0 & 0 \\
0                 &        0           & ... & 0 \\  
0                 &        0           & 0 & J_{2}(\beta_{16})  
\end{array}
\right]
\end{eqnarray*}

For the above matrix $A$, let $C \in \rnn$ be $C = \{{\rm diag}(A^{\trans}A)\}^{-1}$.
Then, $C$ is symmetric positive definite.
In this section, let both $\rho$ and $\gamma$ be $12$.

Fig. \ref{svgp} shows the distribution of the singular values of $A$.
Here, $\sigma_{i}(A) (i=1,2,...,64)$ are the nonzero singular values of $A$.

\begin{figure}[htbp]
\begin{center}
\includegraphics[scale=0.55]{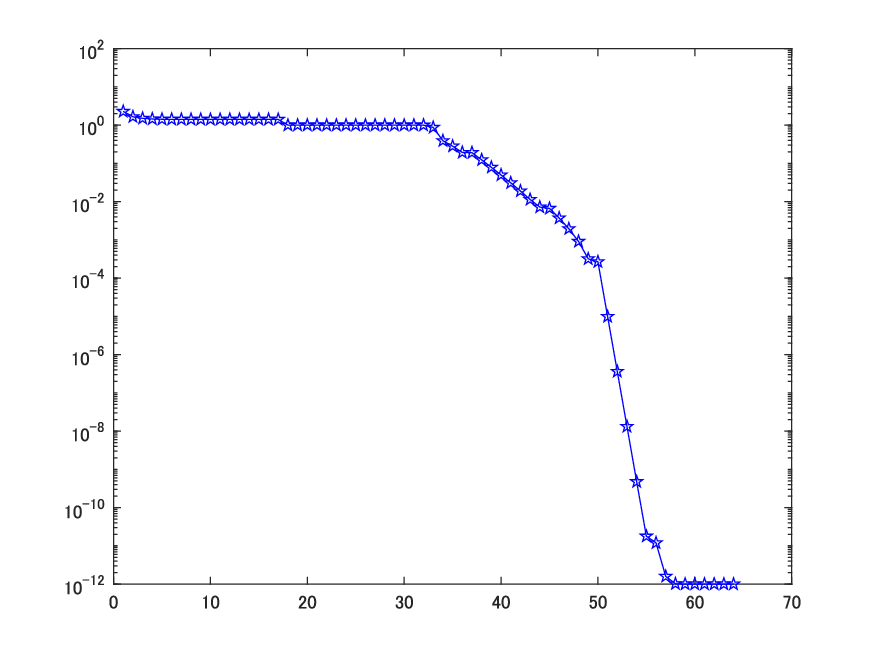}
\caption{Distribution of the singular values $\sigma_{i}(A) (i=1,2,...,64)$ of $A$}
\label{svgp}
\end{center}
\end{figure}

\subsection{Numerical experiments for GP inconsistent systems}
The right hand side vector $\vector{b}$ was set as follows, where $\vector{u}(0,1)$ 
is an $n$ dimensional vector of pseudorandom numbers generated according to the uniform 
distribution over the range (0,1).
\begin{itemize}
\item $\vector{b} = \frac{A \times (1,1,.,1)^{\trans}}{\|A \times (1,1,...,1)^{\trans}\|_{2}}
 + \frac{\vector{u}(0,1)}{\|\vector{u}(0,1)\|_{2}}\times 0.01$
\end{itemize}
Thus, the systems are generically inconsistent.

Here, let $\vector{r} = \vector{b} - A\vector{x}$, where $\vector{x} = A^{\trans} \vector{z}$.
$\displaystyle \min_{\vector{z} \in \rn} \|AA^{\trans}\vector{z} - \vector{b}\|_{2}$
is equivalent to $AA^{\trans}\vector{r} = \vector{0}$.
Furthermore, let $\vector{x} = ACA^{\trans}\vector{z}$. 
$\displaystyle \min_{\vector{z} \in \rn} \|ACA^{\trans}\vector{z} - \vector{b}\|_{2}$
is equivalent to $ACA^{\trans}\vector{r} = \vector{0}$.
Since $\nul(AA^{\trans}) = \nul(ACA^{\trans}) = \nul(A^{\trans})$ holds, 
both \\ $\displaystyle \min_{\vector{z} \in \rn} \|AA^{\trans}\vector{z} - \vector{b}\|_{2}$ and 
$\displaystyle \min_{\vector{z} \in \rn} \|ACA^{\trans}\vector{z} - \vector{b}\|_{2}$ are 
equivalent to $A^{\trans}\vector{r} = \vector{0}$.
Thus, when solving 
$\displaystyle \min_{\vector{z} \in \rn} \|AA^{\trans}\vector{z} - \vector{b}\|_{2}$ and
$\displaystyle \min_{\vector{z} \in \rn} \|ACA^{\trans}\vector{z} - \vector{b}\|_{2}$,
it is appropriate to evaluate the convergence of GMRES by 
$\displaystyle \frac{\|A^{\trans}\vector{r}_{k}\|_{2}}{\|A^{\trans}\vector{r}_{0}\|_{2}}$.
Here, $\vector{r}_{k} = \vector{b} - A\vector{x}_{k}$ for $k = 0, 1, 2,...$.

We compared the performance of GMRES using reorthogonalization and pseudoinverse with
that of GMRES using pseudoinverse for the GP inconsistent systems. As a result, 
the performance of GMRES using reorthogonalization and pseudoinverse was better than
that of GMRES using pseudoinverse.
Thus, in this section, we describe the performance of GMRES using reorthogonalization and 
pseudoinverse.

Fig. \ref{ag_i} 
shows $\displaystyle \frac{\|A^{\trans}\vector{r}_{k}\|_{2}}{\|A^{\trans}\vector{b}\|_{2}}$ versus 
the iteration number $k$ for GMRES using reorthogonalization.
Fig. \ref{aatg_i} and \ref{acatg_i}
show $\displaystyle \frac{\|A^{\trans}\vector{r}_{k}\|_{2}}{\|A^{\trans}\vector{b}\|_{2}}$ versus 
the iteration number $k$ for AB-GMRES using reorthogonalization and pseudoinverse when $B = A^{\trans}$
and $B = CA^{\trans}$, respectively.

As the tolerance value $tol$ for the threshold for the pseudoinverse, 
$10^{-8}\times \sigma_{1}(H)$, $10^{-10} \times \sigma_{1}(H)$, 
$10^{-11}\times \sigma_{1}(H)$ were used as shown
in the captions of the figures. Here, $\sigma_{1}$ is $\sigma_{1}(H_{k+1,k})$.

\begin{figure}[htbp]
\begin{center}
\includegraphics[scale=0.55]{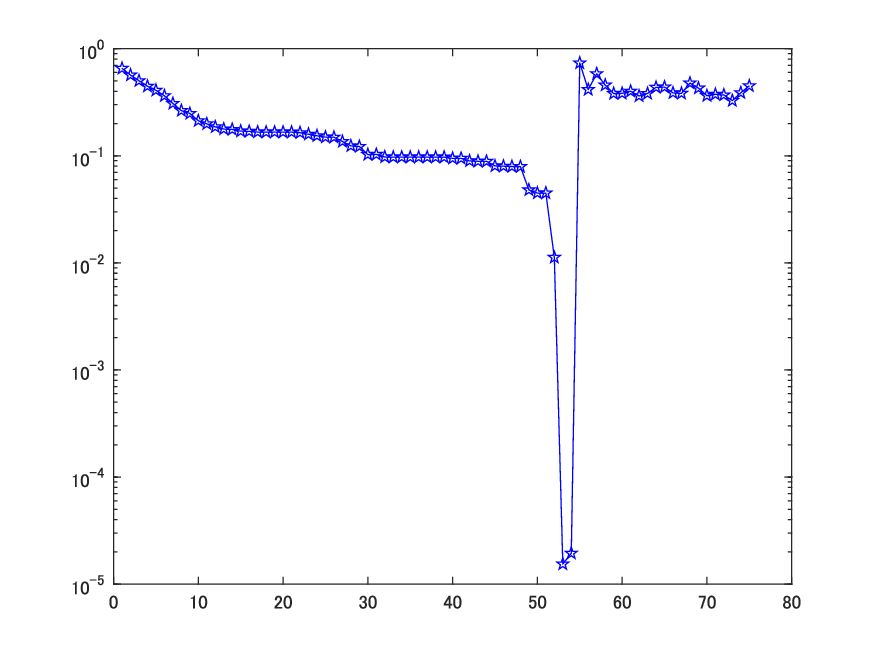}
\caption{$\displaystyle \frac{\|A^{\trans}\vector{r}_{k}\|_{2}}{\|A^{\trans}\vector{b}\|_{2}}$ versus 
the iteration number for GMRES using reorthogonalization}
\label{ag_i}
\end{center}
\end{figure}

\begin{figure}[htbp]
\begin{center}
\includegraphics[scale=0.55]{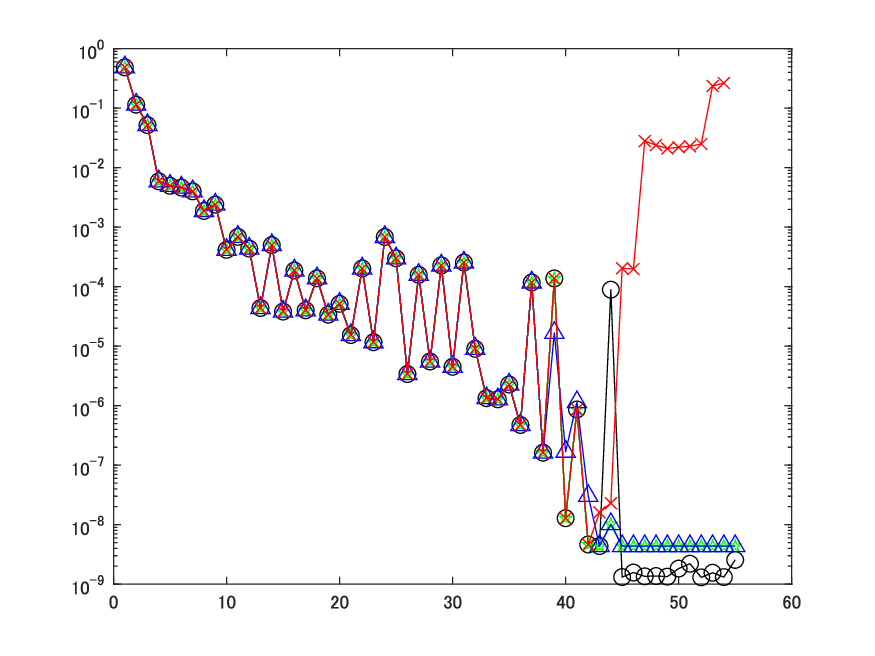}
\caption{$\displaystyle \frac{\|A^{\trans}\vector{r}_{k}\|_{2}}{\|A^{\trans}\vector{b}\|_{2}}$ versus 
the iteration number for AB-GMRES using reorthogonalization and pseudoinverse with 
$10^{-11}\sigma_{1}$ ($\circ$), 
$10^{-8}\sigma_{1}$ ($\triangle$), 
and no pseudoinverse ($\times$) when $B=A^{\trans}$}
\label{aatg_i}
\end{center}
\end{figure}

\begin{figure}[htbp]
\begin{center}
\includegraphics[scale=0.55]{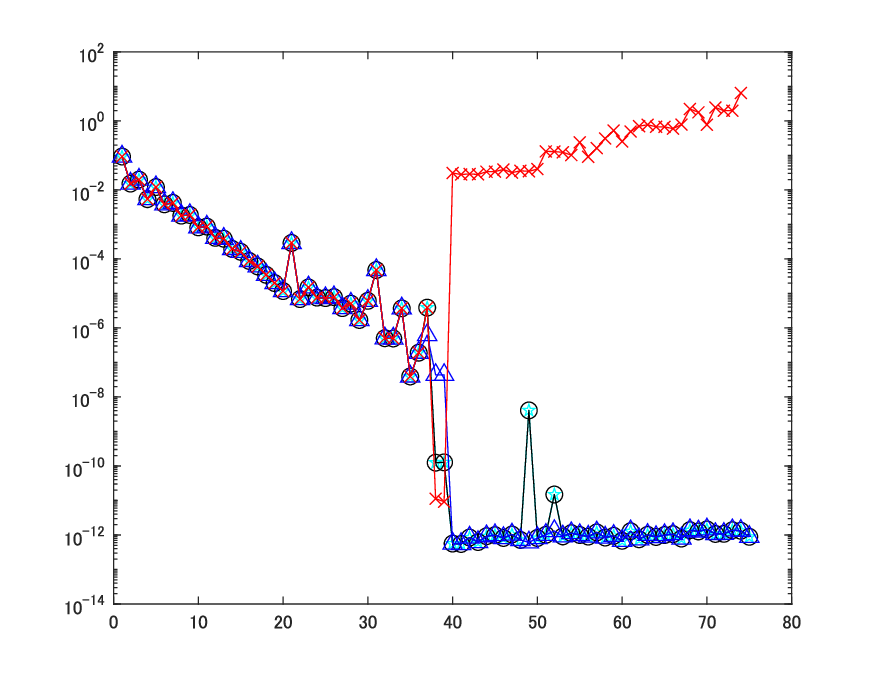}
\caption{$\displaystyle \frac{\|A^{\trans}\vector{r}_{k}\|_{2}}{\|A^{\trans}\vector{b}\|_{2}}$ versus 
the iteration number for AB-GMRES using reorthogonalization and pseudoinverse with
$10^{-11}\sigma_{1}$ ($\circ$),
$10^{-8}\sigma_{1}$ ($\triangle$), and no pseudoinverse ($\times$) when $B=CA^{\trans}$}
\label{acatg_i}
\end{center}
\end{figure}

We observe the following from Fig. \ref{ag_i}, \ref{aatg_i} and \ref{acatg_i}.
\begin{itemize}
\item GMRES using reorthogonalization does not converge for this GP inconsistent system in accordance to 
the convergence theory of GMRES for GP inconsistent systems, while AB-GMRES using 
reorthogonalization and pseudoinverse converges.

\item When we use $10^{-8} \times \sigma_{1}(H_{k+1,k})$
as the tolerance value of pseudoinverse,
the minimum value of $\frac{\|A^{\trans}\vector{r}_{k}\|_{2}}{\|A^{\trans}\vector{b}\|_{2}}$ for AB-GMRES 
using reorthogonalization and
pseudoinverse with preconditioning is $10^{-4}$ times smaller than without preconditioning. 
\end{itemize}

\subsection{Numerical experiments for GP consistent systems}
The right-hand side vectors $\vector{b}$ were set as follows.
\begin{itemize}
\item $\vector{b} = \frac{A \times (1,1,.,1)^{\trans}}{\|A \times (1,1,...,1)^{\trans}\|_{2}}$
\end{itemize}
We evaluate the convergence by $\frac{\|\vector{r}\|_{2}}{\|\vector{b}\|_{2}}$.

Fig. \ref{ag_c}, \ref{aatg_c} and \ref{acatg_c} 
show $\displaystyle \frac{\|\vector{r}_{k}\|_{2}}{\|\vector{b}\|_{2}}$ versus 
the iteration number for GMRES using reorthogonalization, AB-GMRES using reorthogonalization when 
$B = A^{\trans}, CA^{\trans}$.

\begin{figure}[htbp]
\begin{center}
\includegraphics[scale=0.55]{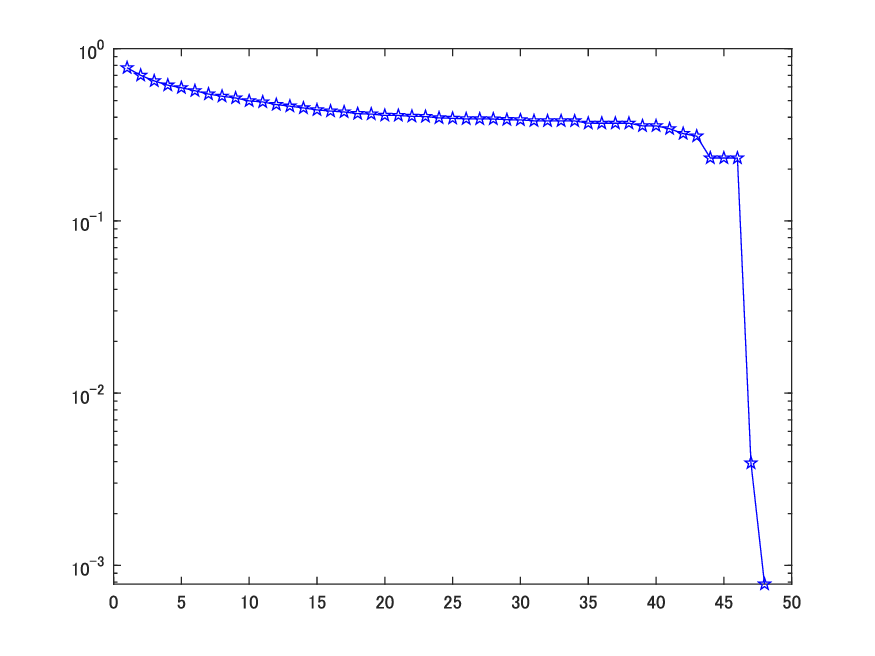}
\caption{$\displaystyle \frac{\|\vector{r}_{k}\|_{2}}{\|\vector{b}\|_{2}}$ versus 
the iteration number for GMRES using reorthogonalization}
\label{ag_c}
\end{center}
\end{figure}

\begin{figure}[htbp]
\begin{center}
\includegraphics[scale=0.55]{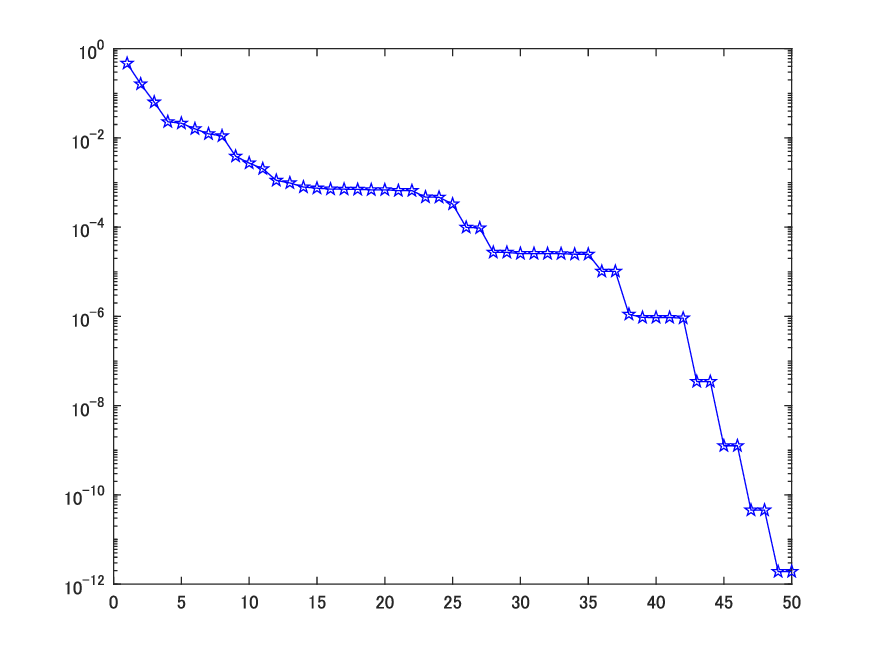}
\caption{$\displaystyle \frac{\|\vector{r}_{k}\|_{2}}{\|\vector{b}\|_{2}}$ versus 
the iteration number for AB-GMRES using reorthogonalization when $B=A^{\trans}$}
\label{aatg_c}
\end{center}
\end{figure}

\begin{figure}[htbp]
\begin{center}
\includegraphics[scale=0.55]{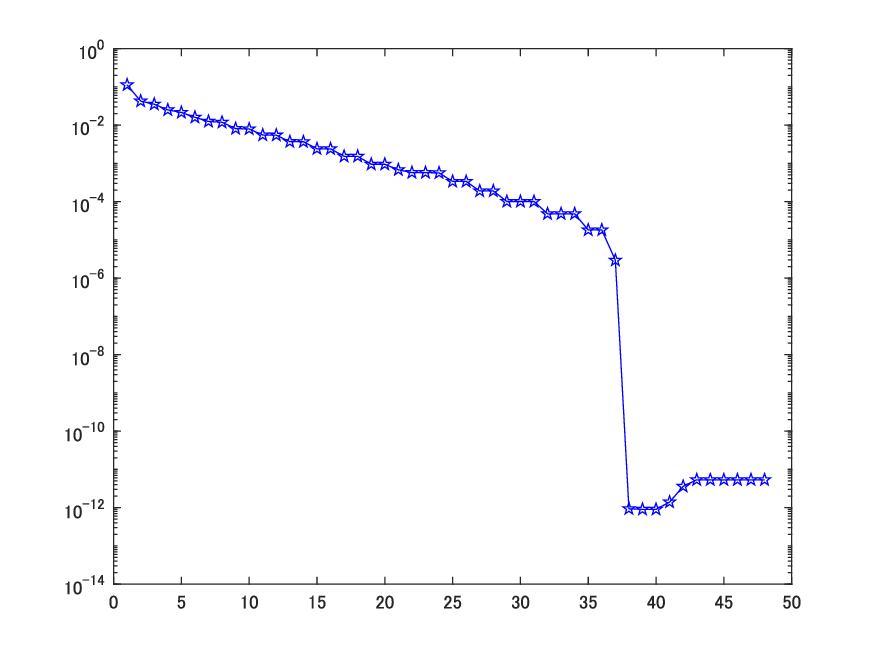}
\caption{$\displaystyle \frac{\|\vector{r}_{k}\|_{2}}{\|\vector{b}\|_{2}}$ versus 
the iteration number for AB-GMRES using reorthogonalization when $B=CA^{\trans}$}
\label{acatg_c}
\end{center}
\end{figure}

In Fig. \ref{ag_c}, GMRES using reorthogonalization breaks down at the 49th iteration 
with $h_{50,49} < 10^{-15}$
after
reaching $\displaystyle \frac{\|\vector{r}_k\|_{2}}{\|\vector{b}\|_{2}} \sim 10^{-3}$ at 
the 48th iteration. 

We observe the following from Fig. \ref{ag_c}, \ref{aatg_c} and \ref{acatg_c}.

\begin{itemize}
\item For this GP consistent system, GMRES using reorthogonalization does not converge unlike
the convergence theory of GMRES for GP consistent systems, 
while AB-GMRES using reorthogonalization converges.

\item $\frac{\|\vector{r}_{k}\|_{2}}{\|\vector{b}\|_{2}}$ of AB-GMRES using 
reorthogonalization when $B = CA^{\trans}$ converges faster than when $B=A^{\trans}$. 
\end{itemize}

\section{Numerical experiments for index 2 systems}\label{NumId2}
The index $2$ matrix $A \in {\bf R}^{128 \times 128}$ is as follows.
\begin{eqnarray*}
\left[
\begin{array}{cc}
A_{11} & A_{12} \\
0                                 &  A_{22}                            
\end{array}
\right].
\end{eqnarray*}
where $A_{11}, A_{12}, A_{22} \in {\bf R}^{64 \times 64}$.
Here, $A_{11}, A_{12}, J_{k}(\lambda), \alpha, \beta$ are the same as those in Section \ref{NumGP}.
The $(2\times i + 63 ,2\times i+ 64)(1 \leq i \leq 16)$th element of $A_{22}$ is $1$, and 
other elements are $0$. Thus, $A$ is an index $2$ matrix.

For the above matrix $A$, let $C \in \rnn$ be $C = \{{\rm diag}(A^{\trans}A)\}^{-1}$.
Then, $C$ is symmetric positive definite.
In this section, let $\rho$ be $12$ and $\gamma$ be $15$.

We compared the performance of AB-GMRES using reorthogonalization and pseudoinverse with
that of AB-GMRES using pseudoinverse for this index $2$ inconsistent systems. Then, 
the performance of AB-GMRES using pseudoinverse was at the same level as
that of AB-GMRES using reorthogonalization and pseudoinverse.
Thus, in this section, we describe the performance of AB-GMRES using  
pseudoinverse without reorthogonalization. 

\subsection{Numerical experiments for index 2 inconsistent systems}\label{NumId2_i}
As GP inconsistent systems, the right-hand side vectors $\vector{b}$ were set as follows, 
where $\vector{u}(0,1)$ 
is an $n$ dimensional vector of pseudorandom numbers generated according to the uniform 
distribution over the range (0,1).
\begin{itemize}
\item $\vector{b} = \frac{A \times (1,1,.,1)^{\trans}}{\|A \times (1,1,...,1)^{\trans}\|_{2}}
 + \frac{\vector{u}(0,1)}{\|\vector{u}(0,1)\|_{2}}\times 0.01$
\end{itemize}

Fig. \ref{aid2_i} 
shows $\displaystyle \frac{\|A^{\trans}\vector{r}_{k}\|_{2}}{\|A^{\trans}\vector{b}\|_{2}}$ versus 
the iteration number $k$ for GMRES.
Fig. \ref{aatid2_i} and \ref{acatid2_i}
show $\displaystyle \frac{\|A^{\trans}\vector{r}_{k}\|_{2}}{\|A^{\trans}\vector{b}\|_{2}}$ versus 
the iteration number $k$ for AB-GMRES using pseudoinverse when $B = A^{\trans}$
and $B = CA^{\trans}$.

\begin{figure}[htbp]
\begin{center}
\includegraphics[scale=0.55]{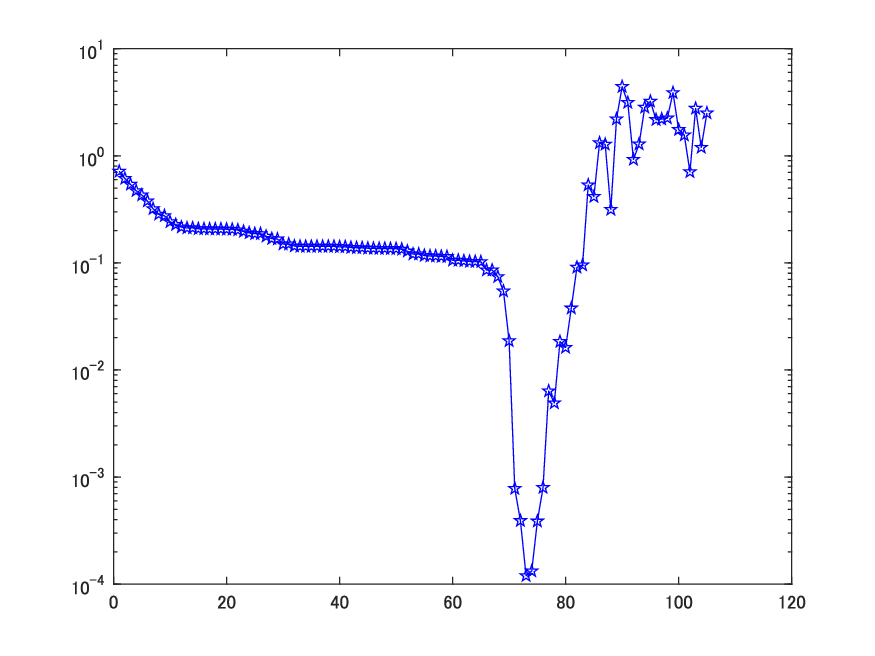}
\caption{$\displaystyle \frac{\|A^{\trans}\vector{r}_{k}\|_{2}}{\|A^{\trans}\vector{b}\|_{2}}$ versus 
the iteration number for GMRES}
\label{aid2_i}
\end{center}
\end{figure}

\begin{figure}[htbp]
\begin{center}
\includegraphics[scale=0.55]{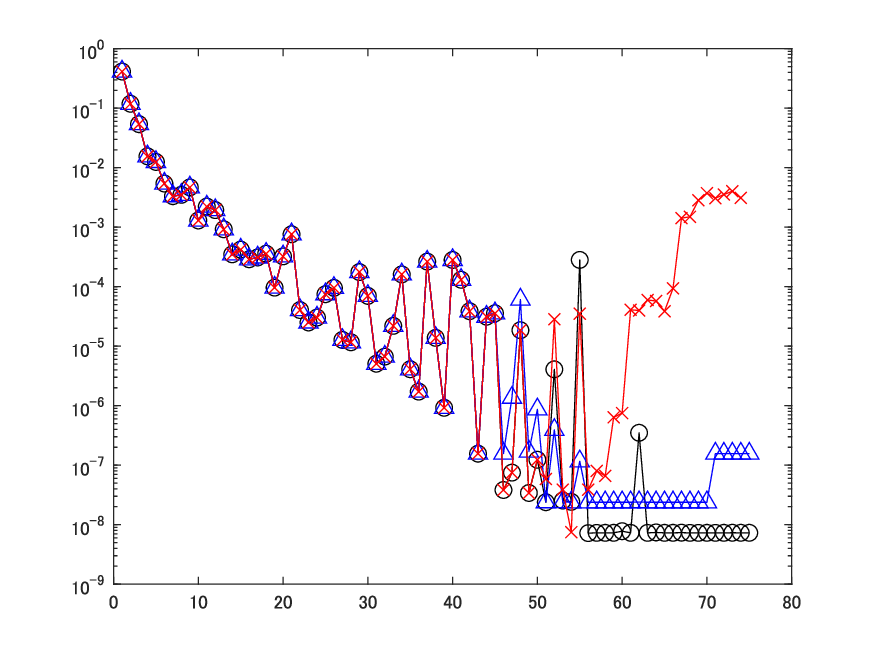}
\caption{$\displaystyle \frac{\|A^{\trans}\vector{r}_{k}\|_{2}}{\|A^{\trans}\vector{b}\|_{2}}$ versus 
the iteration number for AB-GMRES using pseudoinverse with 
$10^{-10}\sigma_{1}$ ($\circ$), $10^{-8}\sigma_{1}$ ($\triangle$), 
and no pseudoinverse ($\times$) when $B=A^{\trans}$}
\label{aatid2_i}
\end{center}
\end{figure}

\begin{figure}[htbp]
\begin{center}
\includegraphics[scale=0.55]{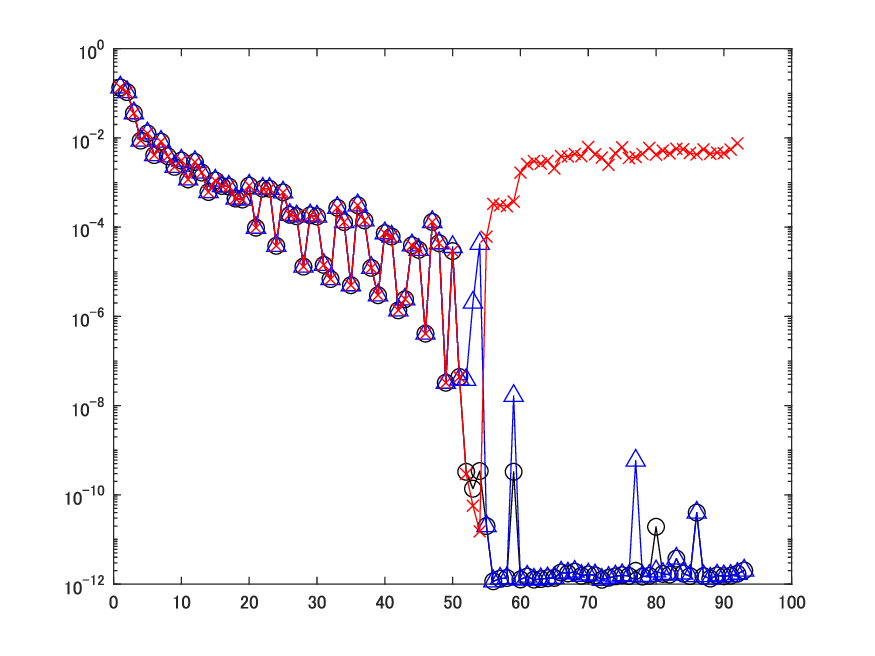}
\caption{$\displaystyle \frac{\|A^{\trans}\vector{r}_{j}\|_{2}}{\|A^{\trans}\vector{b}\|_{2}}$ versus 
the iteration number for AB-GMRES using pseudoinverse with
$10^{-10}\sigma_{1}$ ($\circ$), 
$10^{-8}\sigma_{1}$ ($\triangle$), and no pseudoinverse ($\times$) when $B=CA^{\trans}$}
\label{acatid2_i}
\end{center}
\end{figure}

We observe the following from Fig. \ref{aid2_i}, \ref{aatid2_i} and \ref{acatid2_i}.
\begin{itemize}
\item While GMRES does not converge for this index 2 inconsistent system according to the convergence theory of 
GMRES for GP inconsistent systems, 
AB-GMRES converges. 

\item 
The minimum value of $\frac{\|A^{\trans}\vector{r}_{k}\|_{2}}{\|A^{\trans}\vector{b}\|_{2}}$ for AB-GMRES using
pseudoinverse with preconditioning is at least $10^{-3}$ times smaller than without preconditioning. 
\end{itemize}

\subsection{Numerical experiments for index 2 consistent systems}\label{NumId2_c}
The right-hand side vectors $\vector{b}$ were set as follows.
\begin{itemize}
\item $\vector{b} = \frac{A \times (1,1,.,1)^{\trans}}{\|A \times (1,1,...,1)^{\trans}\|_{2}}$
\end{itemize}
We evaluate the convergence by $\frac{\|\vector{r}\|_{2}}{\|\vector{b}\|_{2}}$.

Fig. \ref{aid2_c} 
shows $\displaystyle \frac{\|\vector{r}_{k}\|_{2}}{\|\vector{b}\|_{2}}$ versus 
the iteration number for GMRES.
Fig. \ref{aatid2_c} and \ref{acatid2_c}
show $\displaystyle \frac{\|\vector{r}_{k}\|_{2}}{\|\vector{b}\|_{2}}$ versus 
the iteration number for AB-GMRES when $B = A^{\trans}$
and $B = CA^{\trans}$.

\begin{figure}[htbp]
\begin{center}
\includegraphics[scale=0.55]{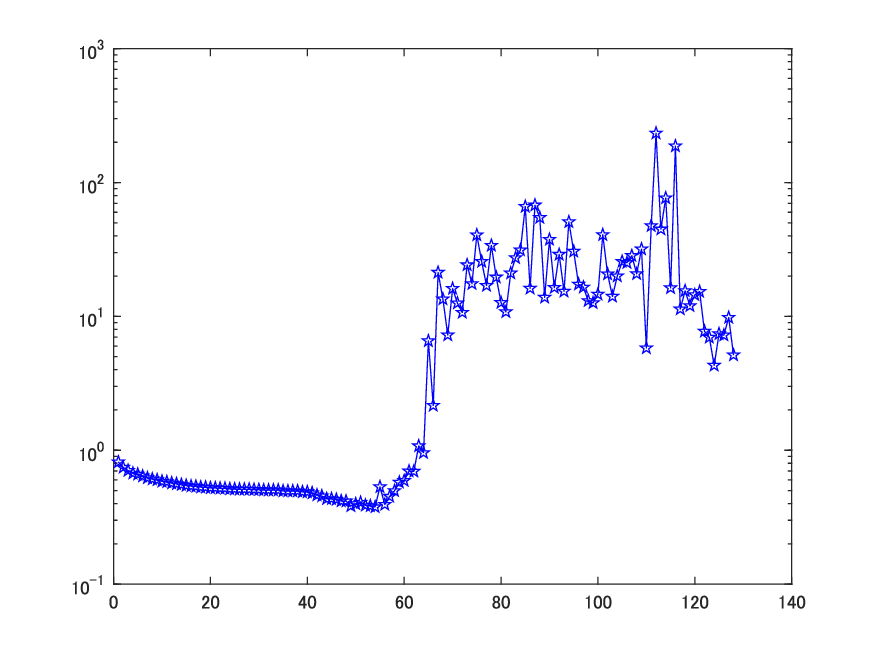}
\caption{$\displaystyle \frac{\|\vector{r}_{k}\|_{2}}{\|\vector{b}\|_{2}}$ versus 
the iteration number for GMRES}
\label{aid2_c}
\end{center}
\end{figure}

\begin{figure}[htbp]
\begin{center}
\includegraphics[scale=0.55]{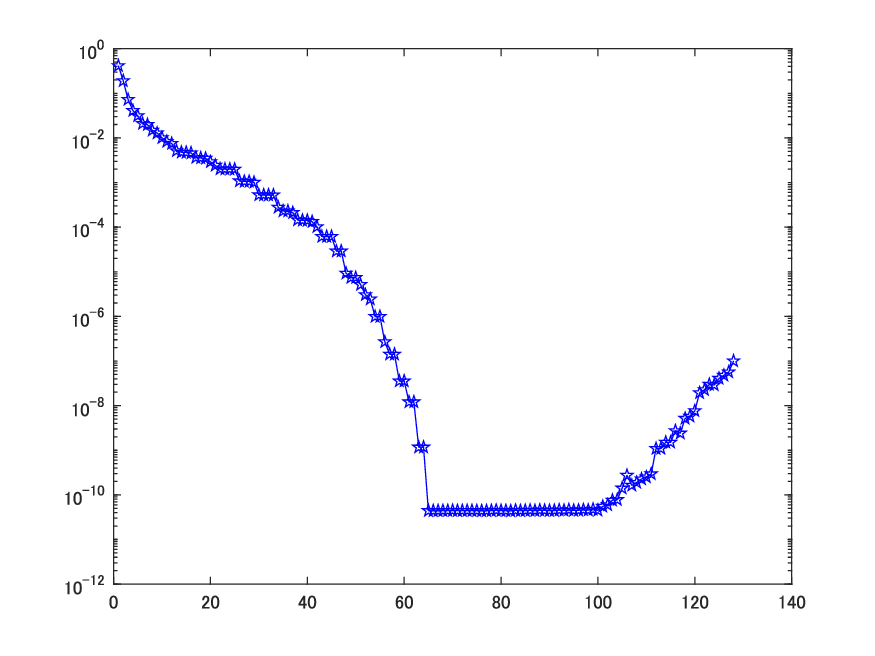}
\caption{$\displaystyle \frac{\|\vector{r}_{k}\|_{2}}{\|\vector{b}\|_{2}}$ versus 
the iteration number for AB-GMRES when $B=A^{\trans}$}
\label{aatid2_c}
\end{center}
\end{figure}

\begin{figure}[htbp]
\begin{center}
\includegraphics[scale=0.55]{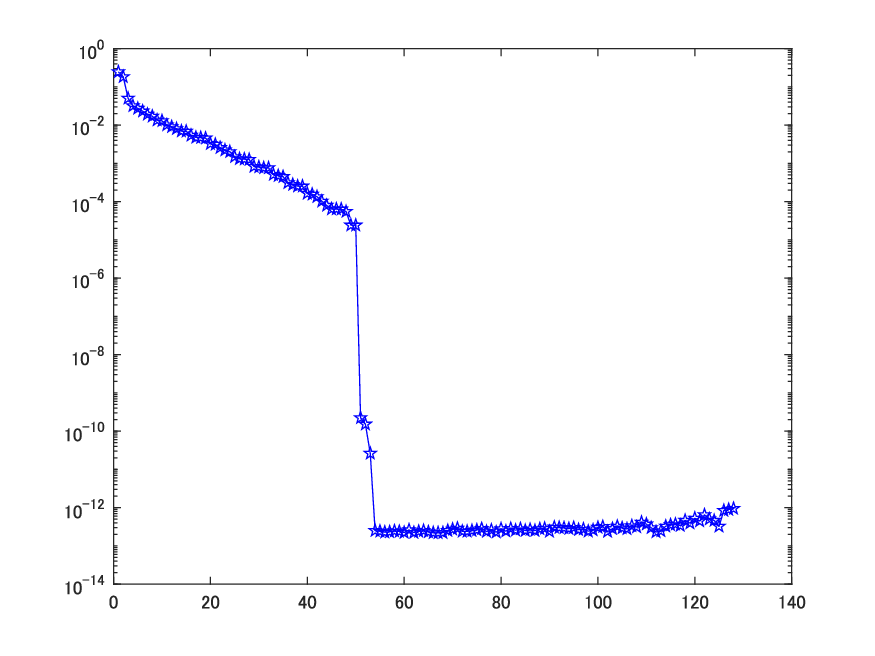}
\caption{$\displaystyle \frac{\|\vector{r}_{k}\|_{2}}{\|\vector{b}\|_{2}}$ versus 
the iteration number for AB-GMRES when $B=CA^{\trans}$}
\label{acatid2_c}
\end{center}
\end{figure}

We observe the following from Fig. \ref{aid2_c}, \ref{aatid2_c} and \ref{acatid2_c}.

\begin{itemize}
\item GMRES does not converge in accordance to
the convergence theory of GMRES for index 2 consistent systems.

\item AB-GMRES with $B=A^{\trans}$ converges.
 
\item The minimum value of $\frac{\|\vector{r}_{k}\|_{2}}{\|\vector{b}\|_{2}}$ for AB-GMRES 
when $B = CA^{\trans}$ is at least $10^{-2}$ times smaller than when $B=A^{\trans}$. 
\end{itemize}

\section{Comparison of left preconditioned GMRES with right preconditioned GMRES}\label{Lfgm}
The normal equations of the first kind $A^{\trans}A\vector{x} = A^{\trans}\vector{b}$ is equivalent to the 
least squares problem (\ref{lstsq2}). 
Then, we will compare the convergence of GMRES applied to $A^{\trans}A\vector{x} = A^{\trans}\vector{b}$
with the convergence of AB-GMRES when $B = A^{\trans}$ and $B = CA^{\trans}$ with $C = \{{\rm diag}(A^{\trans}A)\}^{-1}$.
For numerical experiments, we use the GP consistent and inconsistent systems in Section \ref{NumGP}.
Here, GMRES applied to $A^{\trans}A\vector{x} = A^{\trans}\vector{b}$ is BA-GMRES with $B = A^{\trans}$ \cite{HYI},
where
BA-GMRES is the left preconditioned GMRES.
For $\vector{r}_{k} = \vector{b} - A\vector{x}_{k}$ where $\vector{x}_{k}$ is the $k$th iterate, 
AB-GMRES minimizes $\|\vector{r}_{k}\|_{2}$. On the other hand, BA-GMRES with $B = A^{\trans}$
minimizes $\|A^{\trans}\vector{r}_{k}\|_{2}$.

\subsection{Left preconditioned GMRES for GP inconsistent systems}\label{lfgm_i}
Fig. \ref{alg_i} 
shows $\displaystyle \frac{\|A^{\trans}\vector{r}_{k}\|_{2}}{\|A^{\trans}\vector{b}\|_{2}}$ versus 
the iteration number for BA-GMRES using reorthogonalization when $B = A^{\trans}$.

\begin{figure}[htbp]
\begin{center}
\includegraphics[scale=0.55]{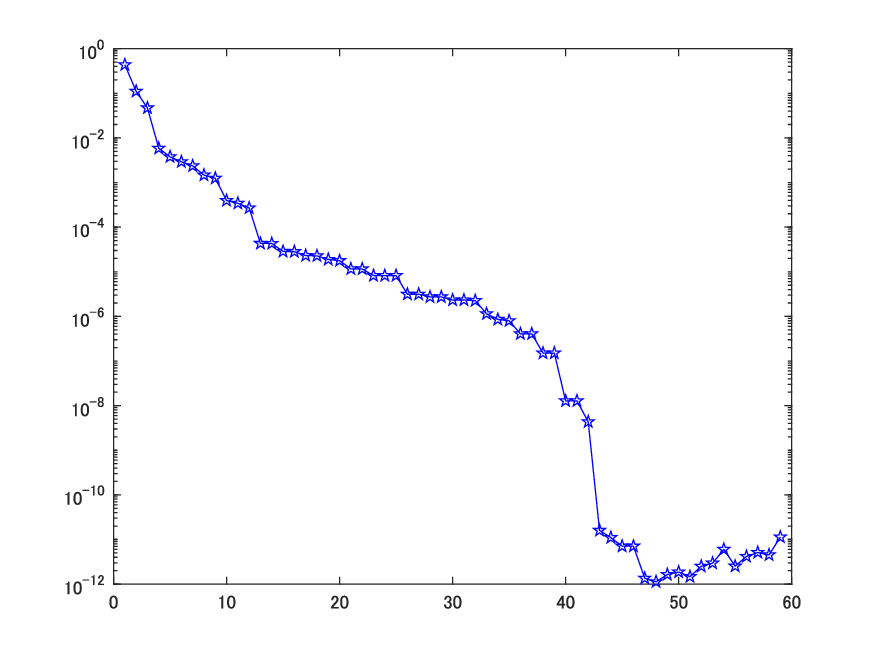}
\caption{$\displaystyle \frac{\|A^{\trans}\vector{r}_{k}\|_{2}}{\|A^{\trans}\vector{b}\|_{2}}$ versus 
the iteration number for BA-GMRES using reorthogonalization when $B = A^{\trans}$}
\label{alg_i}
\end{center}
\end{figure}

We observe the following from Fig. \ref{aatg_i}, \ref{acatg_i} and \ref{alg_i}.
\begin{itemize}
\item $\displaystyle \frac{\|A^{\trans}\vector{r}_{k}\|_{2}}{\|A^{\trans}\vector{b}\|_{2}}$ of BA-GMRES with
$B = A^{\trans}$ is $10^{-3}$ times smaller than that of AB-GMRES using pseudoinverse with $B = A^{\trans}$.
Since AB-GMRES minimizes $\|\vector{r}_{k}\|_{2}$ and BA-GMRES with $B = A^{\trans}$
minimizes $\|A^{\trans}\vector{r}_{k}\|_{2}$, BA-GMRES with $B = A^{\trans}$ is better than AB-GMRES 
with $B=A^{\trans}$ for this inconsistent system.

\item The minimum value of $\displaystyle \frac{\|A^{\trans}\vector{r}_{k}\|_{2}}{\|A^{\trans}\vector{b}\|_{2}}$ of BA-GMRES 
with $B = A^{\trans}$ is almost the same as that of AB-GMRES using pseudoinverse with $B = CA^{\trans}$. While
$\displaystyle \frac{\|A^{\trans}\vector{r}_{k}\|_{2}}{\|A^{\trans}\vector{b}\|_{2}}$ of BA-GMRES increases 
a little when the iteration number is larger than $50$, 
that of AB-GMRES using pseudoinverse with $B = CA^{\trans}$ remains at the minimum level.
Thus, AB-GMRES using pseudoinverse with $B = CA^{\trans}$ is better than BA-GMRES with $B = A^{\trans}$.
We note that BA-GMRES with $B=CA^{\trans}$ does not converge so well.

\end{itemize}

\subsection{Left preconditioned GMRES for GP consistent systems}\label{lfgm_c}
Fig. \ref{alg_c} 
shows $\displaystyle \frac{\|\vector{r}_{k}\|_{2}}{\|\vector{b}\|_{2}}$ versus 
the iteration number for BA-GMRES using reorthogonalization when $B = A^{\trans}$,
for a consistent system.

\begin{figure}[htbp]
\begin{center}
\includegraphics[scale=0.55]{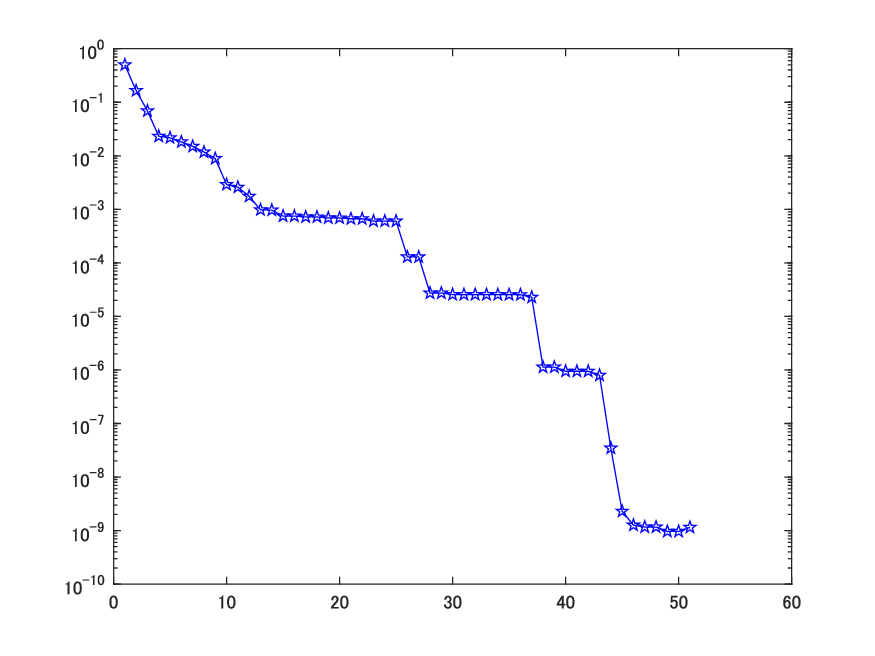}
\caption{$\displaystyle \frac{\|\vector{r}_{k}\|_{2}}{\|\vector{b}\|_{2}}$ versus 
the iteration number for BA-GMRES using reorthogonalization when $B = A^{\trans}$}
\label{alg_c}
\end{center}
\end{figure}

We observe the following from Fig. \ref{aatg_c}, \ref{acatg_c} and \ref{alg_c}.
\begin{itemize}
\item $\displaystyle \frac{\|\vector{r}_{k}\|_{2}}{\|\vector{b}\|_{2}}$ of AB-GMRES using reorthogonalization 
with $B = A^{\trans}$ and $B = CA^{\trans}$ is $10^{-3}$ times smaller than BA-GMRES
using reorthogonalization with $B = A^{\trans}$.
Since AB-GMRES minimizes $\|\vector{r}_{k}\|_{2}$ and BA-GMRES with $B = A^{\trans}$
minimizes $\|A^{\trans}\vector{r}_{k}\|_{2}$, 
AB-GMRES with $B = A^{\trans}$ is better than BA-GMRES 
with $B=A^{\trans}$ for this consistent system.
\end{itemize}


\section*{Acknowledgments}
We would like to thank Dr. Keiichi Morikuni for discussions.

\section{Concluding remarks}\label{sec:ConcL}
We introduced the right preconditioned GMRES for arbitrary singular systems and
proposed using pseudoinverse (pinv of MATLAB) with an appropriate threshold to 
solve the severely ill-conditioned Hessenberg systems arising from 
the Arnoldi process in GMRES for inconsistent systems.
Some numerical experiments on GP and index 2 systems indicate that
the method is effective and robust. We also compared the convergence of the right preconditioned GMRES using pseudoinverse 
with the left preconditioned GMRES, and
showed that the right preconditioned GMRES using pseudoinverse was better for a GP system.
This is only a proof of concept since pinv is expensive.
We are currently investigating more practical ways to implement the idea, as well as
more efficient preconditioners for C.


\end{document}